\documentclass[11pt]{amsart}
\usepackage{amssymb,amsmath,amsthm}
\usepackage[usenames]{color}
\theoremstyle{plain}
\newtheorem{theorem}{Theorem}[section]
\newtheorem{proposition}[theorem]{Proposition}
\newtheorem{lemma}[theorem]{Lemma}
\newtheorem*{lemma*}{Auxiliary Lemma}
\newtheorem{observation}[theorem]{Observation}
\newtheorem{corollary}[theorem]{Corollary}

\newtheorem{conjecture}[theorem]{Conjecture}
\newtheorem{problem}[theorem]{Problem}

\theoremstyle{definition}
\newtheorem{definition}[theorem]{Definition}

\newtheorem{remark}[theorem]{Remark}

\DeclareMathOperator{\lk}{lk}
\DeclareMathOperator{\antist}{ast}

\DeclareMathOperator{\Sn}{\mathfrak{S}}
\DeclareMathOperator{\Pk}{\widehat{\mathfrak{S}}}

\DeclareMathOperator{\pk}{peak}
\DeclareMathOperator{\Des}{Dec}
\DeclareMathOperator{\des}{des}
\DeclareMathOperator{\B}{\mathcal{B}}
\DeclareMathOperator{\w}{\mathbf{w}}

\DeclareMathOperator{\wdn}{\grave{\textit{w}}}
\DeclareMathOperator{\wu}{\acute{\textit{w}}}
\DeclareMathOperator{\ve}{\mathbf{v}}
\DeclareMathOperator{\vd}{\grave{\textit{v}}}
\DeclareMathOperator{\vu}{\acute{\textit{v}}}
\DeclareMathOperator{\ue}{\mathbf{u}}
\DeclareMathOperator{\ud}{\grave{\textit{u}}}
\DeclareMathOperator{\uu}{\acute{\textit{u}}}
\DeclareMathOperator{\Ass}{Assoc}
\DeclareMathOperator{\Cyc}{Cyc}

\DeclareMathOperator{\susp}{susp}

\title{On $\gamma$-vectors satisfying the Kruskal-Katona inequalities}
\author[E. Nevo]{Eran Nevo}
\thanks{Research of the first author partially supported by an NSF Award DMS-0757828.}
\address{Department of Mathematics, Cornell
University, Ithaca USA}
\email{eranevo@math.cornell.edu}
\author[T. K. Petersen]{T. Kyle Petersen}\address{Department of Mathematical Sciences, DePaul University, Chicago USA}
\email{tpeter21@depaul.edu}

\begin{document}
\maketitle
\begin{abstract}
We present examples of flag homology spheres whose $\gamma$-vectors satisfy the Kruskal-Katona inequalities. This includes several families of well-studied simplicial complexes, including Coxeter complexes and the simplicial complexes dual to the associahedron and to the cyclohedron. In these cases, we construct explicit flag simplicial complexes whose $f$-vectors are the $\gamma$-vectors in question, and so a result of Frohmader shows that the $\gamma$-vectors satisfy not only the Kruskal-Katona inequalities but also the stronger Frankl-F\"uredi-Kalai inequalities.  In another direction, we show that if a flag $(d-1)$-sphere has at most $2d+3$ vertices its $\gamma$-vector satisfies the Frankl-F\"uredi-Kalai inequalities. We conjecture that if $\Delta$ is a flag homology sphere then $\gamma(\Delta)$ satisfies the Kruskal-Katona, and further, the Frankl-F\"uredi-Kalai inequalities. This conjecture is a significant refinement of Gal's conjecture, which asserts that such $\gamma$-vectors are nonnegative.
\end{abstract}

\section{Introduction}\label{sec:Intro}

In \cite{Gal} Gal gave counterexamples to the real-root conjecture for flag spheres and conjectured a weaker statement which still implies the Charney-Davis conjecture. The conjecture is phrased in terms of the so-called \emph{$\gamma$-vector}.

\begin{conjecture}[Gal]\cite[Conjecture 2.1.7]{Gal}\label{conj:Gal}
If $\Delta$ is a flag homology sphere then $\gamma(\Delta)$ is nonnegative.
\end{conjecture}

This conjecture is known to hold for the order complex of a Gorenstein$^*$ poset \cite{KaruCD}, all Coxeter complexes (see \cite{Stem}, and references therein), and for the (dual simplicial complexes of the) ``chordal nestohedra" of \cite{PRW}---a class containing the associahedron, permutahedron, and other well-studied polytopes.

If $\Delta$ has a nonnegative $\gamma$-vector, one may ask what these nonnegative integers count. In certain cases (the type A Coxeter complex, say), the $\gamma$-vector has a very explicit combinatorial description. We will exploit such descriptions to show that not only are these numbers nonnegative, but they satisfy certain non-trivial inequalities known as the \emph{Kruskal-Katona inequalities}. Put another way, such a $\gamma$-vector is the $f$-vector of a simplicial complex. Our main result is the following.

\begin{theorem}\label{thm:KrKex}
The $\gamma$-vector of $\Delta$ satisfies the Kruskal-Katona inequalities for each of the following classes of flag spheres:
\begin{itemize}
\item[(a)] $\Delta$ is a Coxeter complex.

\item[(b)] $\Delta$ is the simplicial complex dual to an associahedron.

\item[(c)] $\Delta$ is the simplicial complex dual to a cyclohedron (type B associahedron).
\end{itemize}
\end{theorem}

Note that the type A Coxeter complex is dual to the permutahedron, and for types B and D there is a similarly defined polytope---the ``Coxeterhedron" of Reiner and Ziegler \cite{RZ}.

We prove Theorem \ref{thm:KrKex} by constructing, for each such $\Delta$, a simplicial complex whose faces correspond to the combinatorial objects enumerated by $\gamma(\Delta)$.

In a different direction, we are also able to show that if $\Delta$ is a flag sphere with few vertices relative to its dimension, then its $\gamma$-vector satisfies the Kruskal-Katona inequalities.

\begin{theorem}\label{thm:smallGamma1}
Let $\Delta$ be a $(d-1)$-dimensional flag homology sphere with at most $2d+3$ vertices, i.e., with $\gamma_1(\Delta)\leq 3$.
Then $\gamma(\Delta)$ satisfies the Kruskal-Katona inequalities. Moreover, all possible $\gamma$-polynomials with $\gamma_1\leq 3$ that satisfy the Kruskal-Katona inequalities, except for $1+3t+3t^2$,   occur as
$\gamma(\Delta;t)$ for some flag sphere $\Delta$.
\end{theorem}

Theorem \ref{thm:smallGamma1} is proved by characterizing the structure of such flag spheres.

Computer evidence suggests that Theorems \ref{thm:KrKex} and \ref{thm:smallGamma1} may be enlarged significantly. We make the following strengthening of Gal's conjecture.

\begin{conjecture}\label{conj:flagKrK}
If $\Delta$ is a flag homology sphere then $\gamma(\Delta)$ satisfies the Kruskal-Katona inequalities.
\end{conjecture}

%\kyle{This is a very bold conjecture. I think we should state it in some form, but if we want to we can make the claim softer by writing it within the text, rather than setting it apart in theorem style as it is now.}\eran{let us be bold! (the conj. is true for flag 3-spheres), see below:}
%\eran{changed below. note that if $\gamma_1<2$ it also follows as $\gamma_2$ is an integer.}

This conjecture is true, but not sharp, for flag homology $3$- (or $4$-) spheres. Indeed, Gal showed that $0\leq \gamma_2(\Delta)\leq \gamma_1(\Delta)^2/4$
 must hold for flag homology $3$- (or $4$-) spheres \cite{Gal}, which implies the Kruskal-Katona inequality $\gamma_2(\Delta)\leq \binom{\gamma_1(\Delta)}{2}$.
Our stronger Conjecture \ref{conj:flagFFK} is sharp for flag homology spheres of dimension at most $4$. 
%(However, even Conjecture \ref{conj:flagFFK} is not sharp in general---see Conjecture \ref{prob:flagFlag}). 

In Section \ref{sec:terminology} we review some key definitions. Section \ref{sec:coxnest} collects some known results describing the combinatorial objects enumerated by the $\gamma$-vectors of Theorem \ref{thm:KrKex}. Section \ref{sec:Gam} constructs simplicial complexes based on these combinatorial objects and proves Theorem \ref{thm:KrKex}. Section \ref{sec:small} is given to the proof of Theorem \ref{thm:smallGamma1}.
Finally, Section \ref{sec:FFK} describes a strengthening of Theorem \ref{thm:KrKex} by showing that under the same hypotheses the stronger Frankl-F{\"u}redi-Kalai inequalities hold for the $\gamma$-vector. These inequalities hold in Theorem \ref{thm:smallGamma1} as well, leading us to present a stronger companion to Conjecture \ref{conj:flagKrK}, namely Conjecture \ref{conj:flagFFK}.

\section{Terminology}\label{sec:terminology}

A \emph{simplicial complex} $\Delta$ on a vertex set $V$ is a collection of subsets $F$ of $V$, called \emph{faces}, such that:
\begin{itemize}
\item if $v \in V$ then $\{v\} \in \Delta$,
\item if $F \in \Delta$ and $G \subset F$, then $G \in \Delta$.
\end{itemize}
The dimension of a face $F$ is $\dim F = |F|-1$. In particular $\dim \emptyset=-1$. The dimension of $\Delta$, denoted by $\dim \Delta$, is the maximum of the dimensions of its faces.

 We say that $\Delta$ is \emph{flag} if all the minimal subsets of $V$ which are not in $\Delta$ have size $2$; equivalently $F\in \Delta$ if and only if all the \emph{edges} of $F$ (two element subsets) are in $\Delta$.

We say that $\Delta$ is a sphere if its geometric realization is homeomorphic to a sphere. The \emph{link} $\lk(F) = \lk_{\Delta}(F)$ of a face $F$ of $\Delta$ is the set of all $G \in \Delta$ such that $F \cup G \in \Delta$ and $F \cap G = \emptyset$.
We say that $\Delta$ is a \emph{homology sphere} if for every face $F\in \Delta$, $\lk(F)$ is homologous to the $(\dim \Delta - |F|)$-dimensional sphere. In particular, if $\Delta$ is a sphere then $\Delta$ is a homology sphere.

The \emph{$f$-polynomial} of a $(d-1)$-dimensional simplicial complex $\Delta$ is the generating function for the dimensions of the faces of the complex: \[ f(\Delta;t) := \sum_{F \in \Delta} t^{\dim F + 1} = \sum_{0\leq i \leq d} f_i(\Delta) t^i.\] The \emph{$f$-vector} \[ f(\Delta) := (f_0, f_1,\ldots,f_d) \] is the sequence of coefficients of the $f$-polynomial. We have that $f_i$ is the number of $(i-1)$-dimensional faces of $\Delta$. (We caution the reader that other authors index the $f$-vector as $(f_{-1}, f_0, \ldots, f_{d-1})$, so that $f_i$ is the number of $i$-dimensional faces.)

The \emph{$h$-polynomial} of $\Delta$ is a transformation of the $f$-polynomial:
\[ h(\Delta;t) := (1-t)^d f(\Delta; t/(1-t)) = \sum_{0\leq i\leq d} h_i(\Delta) t^i,\] and the \emph{$h$-vector} is the corresponding sequence of coefficients, \[h(\Delta) := (h_0,h_1,\ldots, h_d).\] Though they contain the same information, often the $h$-polynomial is easier to work with than the $f$-polynomial. For instance, if $\Delta$ is a homology sphere, then the \emph{Dehn-Sommerville relations} guarantee that the $h$-vector is symmetric, i.e., $h_i = h_{d-i}$ for all $0\leq i\leq d$.

When referring to the $f$- or $h$-polynomial of a simple polytope, we mean the $f$- or $h$-polynomial of the boundary complex of its dual. So, for instance, we refer to the $h$-vector of the type A Coxeter complex and the permutahedron interchangeably.

Whenever a polynomial of degree $d$ has symmetric integer coefficients, it has an integer expansion in the basis $\{  t^i(1+t)^{d-2i} : 0\leq i \leq d/2\}$. Specifically, if $\Delta$ is a $(d-1)$-dimensional homology sphere then there exist integers $\gamma_i(\Delta)$ such that \[ h(\Delta;t) = \sum_{0\leq i \leq d/2} \gamma_i(\Delta) t^i (1+t)^{d-2i}.\] We refer to the sequence $\gamma(\Delta) :=(\gamma_0, \gamma_1,\ldots)$ as the \emph{$\gamma$-vector} of $\Delta$, and the corresponding generating function $\gamma(\Delta;t) = \sum \gamma_i t^i$ is the \emph{$\gamma$-polynomial}. Our goal is to show that under the hypotheses of Theorems \ref{thm:KrKex} and \ref{thm:smallGamma1} the $\gamma$-vector for $\Delta$ is seen to be the $f$-vector for some other simplicial complex.

A result of Sch\"utzenberger, Kruskal and Katona (all independently), characterizes the $f$-vectors of simplicial complexes as follows. (See \cite[Ch. II.2]{Sta}.) By convention we call the conditions characterizing these $f$-vectors the \emph{Kruskal-Katona inequalities}.

Given a pair of integers $a$ and $i$ there is a unique expansion:
\[ a = \binom{a_i}{i} + \binom{a_{i-1}}{i-1} + \cdots + \binom{a_j}{j},\]
where $a_i > a_{i-1} > \cdots > a_j \geq j.$ With this in mind, define
\[ a^{(i)} = \binom{a_i}{i+1} + \binom{a_{i-1}}{i} + \cdots + \binom{a_j}{j+1}, \quad 0^{( i )} = 0.\]

\begin{theorem}[Katona, Kruskal, Sch\"utzenberger]
An integer vector $(f_0,f_1 \ldots)$ is the $f$-vector of a simplicial complex if and only if:
\begin{enumerate}
\item[(a)] $f_0 =1$,
\item[(b)] $f_i \geq 0$,
\item[(c)] $f_{i+1} \leq f_i^{(i)}$ for $i = 1,2\ldots$.
\end{enumerate}
\end{theorem}

We will use the Kruskal-Katona inequalities directly for Theorem \ref{thm:smallGamma1} and for checking the Coxeter complexes of exceptional type in  part (a) of Theorem \ref{thm:KrKex}. (See Table \ref{table}.) For the remainder of Theorem \ref{thm:KrKex} we construct explicit simplicial complexes with the desired $f$-vectors.

\begin{table}[t]
\begin{centering}
\begin{tabular}{c | c}
$W$ & $\gamma(W)$ \\
\hline \hline
$E_6$ & $(1,1266,7104,3104)$\\
\hline
$E_7$ & $(1,17628,221808,282176)$ \\
\hline
$E_8$ & $(1, 881744, 23045856, 63613184, 17111296)$\\
\hline
$F_4$ & $(1,232,208)$\\
\hline
$G_2$ & $(1,8)$\\
\hline
$H_3$ & $(1,56)$\\
\hline
$H_4$ & $(1,2632,3856)$\\
\hline
$I_2(m)$ & $(1, 2m-4)$\\
\hline
\end{tabular}
\bigskip
\end{centering}
\caption{The $\gamma$-vectors for finite Coxeter complexes of exceptional type.}
\label{table}
\end{table}

\section{Combinatorial descriptions of $\gamma$-nonnegativity}\label{sec:coxnest}

Here we provide combinatorial descriptions (mostly already known) for the $\gamma$-vectors of the complexes described in Theorem \ref{thm:KrKex}.

\subsection{Type A Coxeter complex}\label{sec:A}

We begin by describing the combinatorial objects enumerated by the $\gamma$-vector of the type $A_{n-1}$ Coxeter complex, or equivalently, the permutahedron. (For the reader looking for more background on the Coxeter complex itself, we refer to \cite[Section 1.15]{Humph}; for the permutahedron see \cite[Example 0.10]{Zieg}.)

Recall that a \emph{descent} of a permutation $w=w_1w_2\cdots w_n \in \mathfrak{S}_n$ is a position $i \in [n-1]$ such that $w_i > w_{i+1}$. 
%\eran{Def. below: do you want to allow valley at position $n$? Term "valley" is used later.}
A \emph{peak} (resp. \emph{valley}) is a position $i \in [2,n-1]$ such that $w_{i-1} < w_i > w_{i+1}$ (resp. $w_{i-1} > w_i < w_{i+1}$). We let $\des(w)$ denote the number of descents of $w$, and we let $\pk(w)$ denote the number of peaks. It is well known that the $h$-polynomial of the type $A_{n-1}$ Coxeter complex is expressed as:
\[ h(A_{n-1};t) = \sum_{w \in \mathfrak{S}_n} t^{\des(w)}.\]  Foata and Sch\"utzenberger were the first to  demonstrate the $\gamma$-nonnegativity of this polynomial (better known as the \emph{Eulerian polynomial}), showing $h(A_{n-1};t) = \sum \gamma_i t^i (1+t)^{n-1-2i}$, where $\gamma_i$ = the number of equivalence classes of permutations of $n$ with $i+1$ peaks \cite{FS}. (Two permutations are in the same equivalence class if they have the same sequence of values at their peaks and valleys.) See also Shapiro, Woan, and Getu \cite{SWG} and, in a broader context, Br\"and\'en \cite{Br} and Stembridge \cite{Stem}.

Following Postnikov, Reiner, and Williams \cite{PRW}, we choose the following set of representatives for these classes:
\[ \Pk_n = \{ w \in \mathfrak{S}_n : w_{n-1} < w_n, \mbox{ and if } w_{i-1} > w_i \mbox{ then } w_i < w_{i+1} \}.\] In other words, $\Pk_n$ is the set of permutations $w$ with no double descents and no final descent, or those for which $\des(w) = \pk(0w0)-1$. We now phrase the $\gamma$-nonnegativity of the type $A_{n-1}$ Coxeter complex in this language.

%\eran{missing 2 added in thm 3.1:}

\begin{theorem}[Foata-Sch\"utzenberger]\cite[Th\'eor\`eme 5.6]{FS}
The $h$-polynomial of the type $A_{n-1}$ Coxeter complex can be expressed as follows:
\[ h(A_{n-1}; t) = \sum_{w \in \Pk_n} t^{\des(w)} (1+t)^{n-1-2\des(w)}.\]
\end{theorem}

We now can state precisely that the type $A_{n-1}$ Coxeter complex (permutahedron) has $\gamma(A_{n-1}) =$ $(\gamma_0, \gamma_1,\ldots, \gamma_{\lfloor \frac{n-1}{2} \rfloor})$, where
\[ \gamma_i(A_{n-1}) = | \{ w \in \Pk_n : \des(w) = i\} |.\]

The permutahedron is an example of a \emph{chordal nestohedron}. Following \cite{PRW}, a chordal nestohedron $P_{\B}$ is characterized by its \emph{building set}, $\B$. Each building set $\B$ on $[n]$ has associated to it a set of \emph{$\B$-permutations}, $\Sn_n(\B) \subset \Sn_n$, and we similarly define $\Pk_n(\B) = \Sn_n(\B) \cap \Pk_n$. See \cite{PRW} for details. The following is a main result of Postnikov, Reiner, and Williams \cite{PRW}.

\begin{theorem}[Postnikov, Reiner, Williams] \label{thm:PRW}
\cite[Theorem 11.6]{PRW} If $\B$ is a connected chordal building set on $[n]$, then
\[ h(P_{\B}; t) = \sum_{w \in \Pk_n(\B)} t^{\des(w)} (1+t)^{n-1-2\des(w)}.\]
\end{theorem}

Thus, for a chordal nestohedron, $\gamma_i(P_{\B}) = |\{  w \in \Pk_n(\B) : \des(w) = i\} |$.

\subsection{Type B Coxeter complex}\label{sec:B}

We now turn our attention to the type $B_n$ Coxeter complex. The framework of \cite{PRW} no longer applies, so we must discuss a new, if similar, combinatorial model.

In type $B_n$, the $\gamma$-vector is given by $\gamma_i$ = $4^i$ times the number of permutations $w$ of $\Sn_n$ such that $\pk(0w)=i$. See Petersen \cite{Pet} and Stembridge \cite{Stem}. We define the set of \emph{decorated permutations} $\Des_n$ as follows. A decorated permutation $\w \in \Des_n$ is a permutation $w \in \Sn_n$ with bars following the peak positions (with $w_0 =0$). Moreover these bars come in four colors: $\{ |=|^0, |^1, |^2, |^3 \}$. Thus for each $w \in \Sn_n$ we have $4^{\pk(0w)}$ decorated permutations in $\Des_n$. For example, $\Des_9$ includes elements such as \[4|238|^1 76519, \quad 4|^3 238 |^2 76519, \quad 25|137 |^1 69 |^2 84 .\] (Note that $\Pk_n \subset \Des_n$.) Let $\pk(\w) = \pk(0w)$ denote the number of bars in $\w$. In this context we have the following result.

\begin{theorem}[Petersen]\cite[Proposition 4.15]{Pet}
The $h$-polynomial of the type $B_n$ Coxeter complex can be expressed as follows:
\[ h(B_n;t) = \sum_{\w \in \Des_n} t^{\pk(\w)}(1+t)^{n-2\pk(\w)}.\]
\end{theorem}

Thus, \[ \gamma_i(B_n) = | \{ \w \in \Des_n : \pk(\w) = i\} |.\]

\subsection{Type D Coxeter complex}\label{sec:D}

We now describe how to view the elements enumerated by the $\gamma$-vector of the type D Coxeter complex in terms of a subset of  decorated permutations. Define a subset $\Des_n^D \subset \Des_n$ as follows:
\begin{align*}
 \Des_n^D = \{ \w = w_1 \cdots |^{c_1} w_{i_1} \cdots &|^{c_2} \cdots \in \Des_n \mbox{ such that } w_1 < w_2 < w_3, \mbox{ or, } \\
& \mbox{ both } \max\{ w_1, w_2, w_3\} \neq w_3 \mbox{ and } c_1 \in \{ 0,1\}
 \}.
\end{align*}
In other words, we remove from $\Des_n$ all elements whose underlying permutations have $w_2 < w_1 < w_3$, then for what remains we dictate that bars in the first or second positions can only come in one of two colors. Stembridge \cite{Stem} gives an expression for the $h$-polynomial of the type $D_n$ Coxeter complex, which we now phrase in the following manner.

\begin{theorem}[Stembridge]\label{thm:stem} \cite[Corollary A.5]{Stem}.
The $h$-polynomial of the type $D_n$ Coxeter complex can be expressed as follows:
\[ h(D_n;t) = \sum_{\w \in \Des_n^D} t^{\pk(\w)}(1+t)^{n-2\pk(\w)}.\]
\end{theorem}

Thus, \[ \gamma_i(D_n) = | \{ \w \in \Des^D_n : \pk(\w) = i\} |.\]

\subsection{The associahedron}\label{sec:ass}

The associahedron $\Ass_n$ is an example of a chordal nestohedron, so Theorem \ref{thm:PRW} applies. Following \cite[Section 10.2]{PRW}, the $\B$-permutations of $\Ass_n$ are precisely the $312$-avoiding permutations. Let $\Sn_n(312)$ denote the set of all $w \in \Sn_n$ such that there is no triple $i < j < k$ with  $w_j < w_k < w_i$. Then we have: \[ h(\Ass_n;t) = \sum_{ w \in \Pk_n(312)} t^{\des(w)} (1+t)^{n-1-2\des(w)},\] where $\Pk_n(312) = \Sn_n(312) \cap \Pk_n$. Hence,  \[ \gamma_i(\Ass_n) = | \{ w \in \Pk_n(312) : \des(w) = i\} |.\]

\subsection{The cyclohedron}

The cyclohedron $\Cyc_n$, or type B associahedron, is a nestohedron, though not a chordal nestohedron and hence Theorem \ref{thm:PRW} does not apply. Its $\gamma$-vector can be explicitly computed from its $h$-vector as described in \cite[Proposition 11.15]{PRW}. We have $\gamma_i(\Cyc_n) = \binom{n}{i,i,n-2i}$. Define \[ P_n = \{ (L,R) \subseteq [n] \times [n] : |L|=|R|, L \cap R = \emptyset\}.\] It is helpful to think of elements of $P_n$ as follows. For $\sigma = (L, R)$ with $|L|=|R| = k$, write $\sigma$ as a $k\times 2$ array with the elements of $L$ written in increasing order in the first column, the elements of $R$ in increasing order in the second column. That is, if $L = \{ l_1 < \cdots < l_k\}$ and $R = \{ r_1 < \cdots < r_k\}$, we write \[ \sigma = \left( \begin{array}{c c}
           l_1 & r_1 \\
           \vdots & \vdots \\
           l_k & r_k \end{array} \right). \]
For $\sigma \in P_n$, let $\rho(\sigma) = |L|=|R|$. Then we can write
\[ h(\Cyc_n;t) = \sum_{\sigma \in P_n} t^{\rho(\sigma)} (1+t)^{n-2\rho(\sigma)}.\]

Thus, \[ \gamma_i(\Cyc_n) = | \{ \sigma \in P_n : \rho(\sigma) = i\} |.\]

\section{The $\Gamma$-complexes}\label{sec:Gam}

We will now describe simplicial complexes whose $f$-vectors are the $\gamma$-vectors described in Section \ref{sec:coxnest}.

%First we introduce the complex $\Gamma(\Des_n)$, whose faces are in bijection with the elements of $\Des_n$, such that if $F \leftrightarrow \w$, then $\dim F = \pk(\w)-1$. 

%\kyle{I like the notation $\Gamma(*)$ since it has connotations both with the (lowercase) $\gamma$-vector we're modeling and with the usual use of $\Gamma$ to denote a graph - and these complexes are clique complexes}

\subsection{Coxeter complexes}

Notice that if \[ \w = w_1 |^{c_1} \cdots |^{c_{i-1}} w_i |^{c_i} w_{i+1} |^{c_{i+1}} \cdots |^{c_{l-1}} w_l, \] is a decorated permutation, then each word $w_i = w_{i,1} \ldots w_{i,k}$ has some $j$ such that: \[ w_{i,1} > w_{i,2} > \cdots > w_{i,j} > w_{i,j+1} <  w_{i,j+2} < \cdots < w_{i,k}.\] We say $w_i$ is a \emph{down-up word}. We call $\wdn_i= w_{i,1} \cdots w_{i,j}$ the \emph{decreasing part} of $w_i$ and $\wu_i = w_{i,j+1} \cdots w_{i,k}$ the \emph{increasing part} of $w_i$. Note that the decreasing part may be empty, whereas the increasing part is nonempty if $i \neq l$. Also, the rightmost block of $\w$ may be strictly decreasing (in which case $w_l=\wdn_l$) and the leftmost block is always increasing, even if it is a singleton.

Define the vertex set \[ V_{\Des_n}:= \{ \ve \in \Des_n : \pk(\ve) = 1\}.\] The adjacency of two such vertices is defined as follows. Let \[ \ue = \uu_1 |^c \ud_2 \uu_2\] and \[ \ve = \vu_1 |^d \vd_2 \vu_2\] be two vertices with $|\uu_1| < |\vu_1|$. We define $\ue$ and $\ve$ to be adjacent if and only if there is an element $\w \in \Des_n$ such that \[ \w = \uu_1 |^c \ud_2 \acute{a} |^d \vd_2 \vu_2,\] where $\acute{a}$ is the letters of $\uu_2 \cap \vu_1$ written in increasing order. Such an element $\w$ exists if, as sets: 
\begin{itemize}
\item $\uu_1\cup \ud_2 \subset \vu_1$ ($\Leftrightarrow \vd_2\cup \vu_2 \subset \uu_2$),
\item $\min \uu_2 \cap \vu_1 < \min \ud_2$, and
\item $\max \uu_2 \cap \vu_1 > \max \vd_2$.  
(Note that $\uu_2 \cap \vu_1$ is nonempty by the first condition.)
\end{itemize}

\begin{definition}
Let $\Gamma(\Des_n)$ be the collection of all subsets $F$ of $V_{\Des_n}$ such that every two distinct vertices in $F$ are adjacent.
\end{definition}

Note that by definition $\Gamma(\Des_n)$ is a flag complex. It remains to show that the faces of $\Gamma(\Des_n)$ correspond to decorated permutations.

Let $\phi: \Des_n \to \Gamma(\Des_n)$ be the map defined as follows. If
 \[ \w = w_1 |^{c_1} \cdots |^{c_{i-1}} w_i |^{c_i} w_{i+1} |^{c_{i+1}} \cdots |^{c_{l-1}} w_l, \] then
 \[ \phi(\w) = \{ w_1|^{c_1} \wdn_2 \acute{b}_1, \ldots, \acute{a}_i|^{c_i} \wdn_{i+1} \acute{b}_i, \ldots, \acute{a}_{l-1}|^{c_{l-1}} \wdn_l \acute{b}_{l-1}\},\]
where $\acute{a}_i$ is the set of letters to the left of $\wdn_{i+1}$ in $\w$ written in increasing order and $\acute{b}_i$ is the set of letters to the right of $\wdn_{i+1}$ in $\w$ written in increasing order.

\begin{proposition}\label{prop:SetBijection}
The map $\phi$ is a bijection between faces of $\Gamma(\Des_n)$ and decorated permutations in $\Des_n$.
\end{proposition}

\begin{proof}
%\eran{added well-definedness, and changed rest of proof. In particular, strengthened induction hypothesis for surjectiveness. Hope you'll find it clear.}
First, let let us check that $\phi$ is well defined, i.e. that $\phi(\w)\in \Gamma(\Des_n)$ for $\w \in \Des_n$. Indeed, it is easy to verify that the three bulleted conditions above hold for any two vertices in $\phi(\w)$. 

It is straightforward to see that $\phi$ is injective. 
Indeed, if $\phi(\w)=\phi(\ve)$ then $\w$ and $\ve$ have the like colored bars in the same positions. Further, their vertex with bar $|^{c_1}$ shows $w_1=v_1$ and $\wdn_2=\vd_2$. Therefore, their vertex with bar $|^{c_2}$  shows $\wu_2=\vu_2$ and $\wdn_3=\vd_3$, and inductively, $\w=\ve$.

%(Indeed, if $\w$ and $\ve$ are distinct yet have the like colored bars in the same positions, then there is some $k \in [n]$ such that there is a bar between where $k$ appears in $\w$ and where $k$ appears in $\ve$. This results in a different vertex set.)

To see that $\phi$ is surjective, we will construct the inverse map. Clearly if $|F|\leq 2$, there is an element of $\Des_n$ corresponding to $F$. Now, given any $F \in \Gamma(\Des_n)$, order its vertices  by increasing position of the bar in the vertex: $\ue_1, \ue_2,\ldots, \ue_l$. Suppose by induction on $|F|$ that the face $\{ \ue_1, \ldots, \ue_{l-1}\}$ corresponds to the decorated permutation \[ \w = w_1 |^{c_1} \cdots |^{c_{i-1}} w_i |^{c_i} w_{i+1} |^{c_{i+1}} \cdots |^{c_{l-1}} \wdn_l \wu_l,\] so that  
$\ue_{l-1} = \uu_{l-1}|^{c_{l-1}} \wdn_l \wu_l$.

Then since $\ue_{l-1}$ and $\ue_l = \uu_{l,1}|^{c_l}\ud_{l,2}\uu_{l,2}$ are adjacent, we know $\ud_{l,2} \cup \uu_{l,2} \subset \wu_l$, $\min \wu_l \cap \uu_{l,1} < \min \wdn_l$, and $\max \wu_l \cap \uu_{l,1} > \max \ud_{l,2}$. Then obviously the following is in fact a decorated permutation in $\Des_n$: \[ \w'= w_1 |^{c_1} \cdots |^{c_{i-1}} w_i |^{c_i} w_{i+1} |^{c_{i+1}} \cdots |^{c_{l-1}} \wdn_l \acute{a} |^{c_l} \ud_{l,2} \uu_{l,2},\] where $\acute{a} = \wu_l \cap \uu_{l,1}$ written in increasing order. By construction, we have $\phi(\w') = F$, completing the proof.
\end{proof}

We now make explicit how to realize $\Des_n$ as the face poset of $\Gamma(\Des_n)$. We say $\w$ covers $\ue$ if and only if $\ue$ can be obtained from $\w$ by removing a bar $|^{c_i}$ and reordering the word $w_i w_{i+1} = \wdn_i\wu_i w_{i+1}$ as a down-up word $\wdn_i a$ where $a$ is the word formed by writing the letters of $\wu_i w_{i+1}$ in increasing order. Then $(\Des_n,\leq)$ is a poset graded by number of bars.

\begin{proposition}\label{prop:PosetBijection}
The map $\phi$ is an isomorphism of graded posets from  
$(\Des_n,\leq)$ to $(\Gamma(\Des_n),\subseteq)$.
\end{proposition}

\begin{proof}
The previous proposition shows the map $\phi$ is a grading-preserving bijection. We verify that $\phi$  and $\phi^{-1}$ are order preserving. If $\w\leq \ve$ then clearly $\phi(\w)\subseteq \phi(\ve)$ for both the bars in $\w$ and their adjacent decreasing parts are unaffected by the removal of other bars from $\ve$. 

If $G=F \cup \{\ue \}$ is in $\Gamma(\Des_n)$, we now show that $\phi^{-1}(F)\leq \phi^{-1}(G)$. For $|G|\leq 2$ this is obvious. The general situation follows from showing that $\phi^{-1}(G)$ is independent of the order in which its bars are inserted. More precisely, it is enough to check that for three pairwise adjacent vertices $\ue = \uu_1|^c \ud_2 \uu_2, \ve =\vu_1|^d \vd_2 \vu_2$, and $\w=\wu_1|^e \wdn_2 \wu_2$ in $V_{\Des_n}$ (with bars in increasing position order $|^{c},|^{d},|^{e}$ respectively,) we can insert the middle bar last. This can be done if the following holds: 
\[\uu_2\cap \wu_1 = (\uu_2\cap \vu_1)\cup \vd_2 \cup (\vu_2\cap \wu_1).\]
Equality holds since in both $\phi^{-1}(\{\ue,\w\})$ and $\phi^{-1}(\{\ue,\ve,\w\})$ the words to the right of $|^e$ and to the left of $|^c$ are the same.
\end{proof}

We now show that the $\gamma$-objects for the type $A_{n-1}$ and type $D_n$ Coxeter complexes form flag subcomplexes of $\Gamma(\Des_n)$.

\begin{proposition}
For $S \in \{\Pk_n, \Des_n^D\}$ the image $\Gamma(S):= \phi(S)$ is a flag subcomplex of $\Gamma(\Des_n)$. 
\end{proposition}

\begin{proof}
%\eran{proof changed, as I think that what you showed is not enough.}
To show $\Gamma(S)$ is a subcomplex, by Proposition \ref{prop:PosetBijection} it suffices to show that 
%for $\w \in S$, the vertices of $\phi(\w)$ are contained in: \[ V_S := \{ \ve \in S : \ve \mbox{ has exactly one bar}\}.\] This is straightforward to verify in all cases.
$(S,\leq)$ is a lower ideal in $(\Des_n,\leq)$. This is straightforward to verify in all cases.

For $\w \in \Pk_n$, all bars have color 0 and all subwords between bars are increasing. Omitting a bar $|^{c_i}$ from $\w$ we reorder $w_iw_{i+1}$ in increasing order as $\wdn_i$ is empty, thus the resulting element is in $\Pk_n$. 

Finally, if $\w \in \Des_n^D$, we observe that if the first three letters of $\w$ do not satisfy $w_2 < w_1 < w_3$, then neither can the first three letters of any coarsening of $\w$. 

To show that $\Gamma(S)$ is flag, we will show that it is the flag complex generated by the elements of $S$ with exactly one bar. Precisely, let \[ V_S := \{ \ve \in S : \ve \mbox{ has exactly one bar}\}.\] Since we have already shown $S$ is a lower ideal we know if $\w \in S$, $\phi(S) \subset V_S$. It remains to show that if $F$ is a collection of pairwise adjacent vertices in $V_S$ then we have $\phi^{-1}(F) \in S$. (Pairwise adjacency guarantees $\phi^{-1}(F)$ is well-defined; suppose each $F$ below has this property.) We now examine the combinatorics of each case individually.

First, if $F \subset V_{\Pk_n}$, then all the vertices of $F$ are of the form $\wu_1 | \wu_2$, and so $\phi^{-1}(F)$ has only 0-colored bars and no decreasing parts. That is, $\phi^{-1}(F) \in \Pk_n$.

In the case of $\Des_n^D$, observe that $\w \in \Des_n$ has $w_2 < w_1 < w_3$ if and only if $\phi(\w)$ has a vertex with the same property, and likewise for the color of a bar in position 1 or 2. Thus if $F \subset V_{\Des_n^D}$, then because each vertex avoids $w_2 < w_1 < w_3$ and has appropriately colored bars (if any) in positions 1 and 2, we have $\phi^{-1}(F) \in \Des_n^D$. This completes the proof.
\end{proof}

In light of the results of Sections \ref{sec:A}, \ref{sec:B}, and \ref{sec:D}, and because the dimension of faces corresponds to the number of bars, we have the following result, which, along with Table \ref{table} implies part (a) of Theorem \ref{thm:KrKex}.

\begin{corollary}\label{cor:ABDgam}
We have:
\begin{enumerate}
\item $\gamma(A_{n-1}) = f(\Gamma(\Pk_n))$,
\item $\gamma(B_n)=f(\Gamma(\Des_n)$, and
\item $\gamma(D_n)=f(\Gamma(\Des_n^D)$.
\end{enumerate}
In particular, the $\gamma$-vectors of the type $A_{n-1}$, $B_n$, and $D_n$ Coxeter complexes satisfy the Kruskal-Katona inequalities.
\end{corollary}

\begin{remark}
The construction of $\Gamma(\Des_n)$ admits an obvious generalization to any number of colors of bars, though we have no examples of simplicial complexes whose $\gamma$-vectors would be modeled by the faces of such a complex (and for which a result like Corollary \ref{cor:ABDgam} might exist). 
\end{remark}

\begin{remark}
In view of Theorem \ref{thm:PRW}, we can observe that if $\B$ is a connected chordal building set such that $(\Pk_n(\B),\leq)$ is a lower ideal in $(\Des_n,\leq)$, then a result such as Corollary \ref{cor:ABDgam} applies. That is, we would have $\gamma(P_{\B}) = f(\phi(\Pk_n(\B)))$. In particular, we would like to use such an approach to the $\gamma$-vector of the associahedron. However, $\Pk_n(312)$ is not generally a lower ideal in $\Des_n$. For example, with $n=5$, we have $w=3|24|15 > 3|1245=u$. While $w$ is 312-avoiding, $u$ is clearly not.
\end{remark}

\subsection{The associahedron}

First we give a useful characterization of the set $\Pk_n(312)$.

\begin{observation}\label{ob:312}
If $w \in \Pk_n(312)$, it has the form 
\begin{equation}\label{eq:312}
 w = \acute{a}_1 \,\, j_1 i_1 \,\, \acute{a}_2 \,\, j_2 i_2 \, \cdots \, \acute{a}_k \,\, j_k i_k \,\, \acute{a}_{k+1},
\end{equation}
 where:
\begin{itemize}
\item $j_1 < \cdots < j_k$,
\item $j_s > i_s$ for all $s$, and 
\item $\acute{a}_s$ is the word formed by the letters of $\{ r \in [n] \setminus \{ i_1, j_1, \ldots, i_k,j_k\} : j_{s-1} < r < j_s \}$ (with $j_0 = 0$, $j_{k+1} = n+1$) written in increasing order.
\end{itemize}
\end{observation}

In particular, since $w$ has no double descents and no final descent, we see that $\acute{a}_{k+1}$ is always nonempty and $w_n = n$. We refer to $(i_s,j_s)$ as a \emph{descent pair} of $w$.

Given distinct integers $a,b,c,d$ with $a<b$ and $c<d$, we say the pairs $(a,b)$ and $(c,d)$ are \emph{crossing} if either of the following statements are true:
\begin{itemize}
\item $a < c < b < d$ or
\item $c < a < d < b$.
\end{itemize}
Otherwise, we say the pairs are \emph{noncrossing}. For example, $(1, 5)$ and $(4,7)$ are crossing, whereas both the pairs $(1,5)$ and $(2,4)$ and the pairs $(1,5)$ and $(6,7)$ are noncrossing.

Define the vertex set \[ V_{\Pk_n(312)} := \{ (a,b) : 1 \leq a < b \leq n-1 \}.\]

\begin{definition}
Let $\Gamma(\Pk_n(312))$ be the collection of subsets $F$ of $V_{\Pk_n(312)}$ such that every two distinct vertices in $F$ are noncrossing.
\end{definition}

By definition $\Gamma(\Pk_n(312))$ is a flag simplicial complex, and so the task remains to show that the faces of the complex correspond to the elements of $\Pk_n(312)$.

Define a map $\pi: \Pk_n(312) \to \Gamma(\Pk_n(312))$ as follows: \[ \pi(w) = \{ (w_{i+1}, w_i) : w_i > w_{i+1} \}.\]

\begin{proposition}
The map $\pi$ is a bijection between faces of $\Gamma(\Pk_n(312))$ and $\Pk_n(312)$.
\end{proposition}

\begin{proof}
Suppose $w$ is as in \eqref{eq:312}. We claim that the descent pairs $(i_s,j_s)$ and $(i_t,j_t)$ (with $j_s<j_t$, say) are noncrossing. Indeed, if $i_s < i_t < j_s < j_t$, then the subword $j_s i_s i_t$ forms the pattern 312. Therefore (and because $w_n=n$) we see the map $\pi(w) = \{ (i_1, j_1), \ldots, (i_k,j_k)\}$ is well-defined.

That $\pi$ is injective follows from the Observation \ref{ob:312}. Indeed if $\pi(w) = \pi(v)$, then because $j_1 < \cdots < j_k$ the descents $j_s i_s$ occur in the same relative positions in $w$ as in $v$, and the contents of the increasing words $\acute{a}_s$ are forced after identifying the descent pairs, then $w=v$.

Now consider a face $F = \{ (i_1,j_1),\ldots, (i_k,j_k)\}$ of $\Gamma(\Pk_n(312))$. To construct $\pi^{-1}(F)$, we simply order the pairs in $F$ so that $j_1 < \cdots < j_k$ and form the permutation $\pi^{-1}(F) = w$ as in \eqref{eq:312}.
\end{proof}

By construction, we have $|\pi(w)| = \des(w)$, and therefore the results of Section \ref{sec:ass} imply the following result, proving part (b) of Theorem \ref{thm:KrKex}.

\begin{corollary}
We have:
\[
\gamma(\Ass_n)=f(\Gamma(\Pk_n(312))).
\]
In particular, the $\gamma$-vector of the associahedron satisfies the Kruskal-Katona inequalities.
\end{corollary}

\begin{remark}
It is well known that the $h$-vector of the associahedron has a combinatorial interpretation given by \emph{noncrossing partitions}. Simion and Ullmann \cite{SU} give a particular decomposition of the lattice of noncrossing partitions that can be used to describe $\gamma(\Ass_n)$ in a (superficially) different manner.
\end{remark}

\subsection{The cyclohedron}

For the cyclohedron, let \[ V_{P_n} := \{ (l,r) \in [n] \times [n] : l \neq r \}.\] Two vertices $(l_1,r_1)$ and $(l_2,r_2)$ are adjacent if and only if:
\begin{itemize}
\item $l_1, l_2, r_1, r_2$ are distinct and 
\item $l_1 < l_2$ if and only if $r_1 < r_2$.
\end{itemize}
Define $\Gamma(P_n)$ to be the flag complex whose faces $F$ are all subsets of $V_{P_n}$ such that every two distinct vertices in $F$ are adjacent.

We let $\psi: P_n \to \Gamma(P_n)$ be defined as follows. If \[\sigma = \left( \begin{array}{c c}
           l_1 & r_1 \\
           \vdots & \vdots \\
           l_k & r_k \end{array} \right) \] is an element of $P_n$, then $\psi(\sigma)$ is simply the set of rows of $\sigma$:
\[ \psi(\sigma) = \{ (l_1, r_1), \ldots, (l_k, r_k)\}. \]  Clearly this map is invertible, for we can list a set of pairwise adjacent  vertices in increasing order (by $l_i$ or by $r_i$) to obtain an element of $P_n$. We have the following.

\begin{proposition}
The map $\psi$ is a bijection between faces of $\Gamma(P_n)$ and the elements of $P_n$.
\end{proposition}

We are now able to complete the proof of Theorem \ref{thm:KrKex}, as the following implies part (c).

\begin{corollary}
We have \[ \gamma(\Cyc_n) = f(\Gamma(P_n)).\] In particular, the $\gamma$-vector of the cyclohedron satisfies the Kruskal-Katona inequalities.
\end{corollary}

\section{Flag spheres with few vertices}\label{sec:small}
%\eran{I've drastically revised this section. Please check if it is clear now.}

We now describe a different class of flag spheres whose $\gamma$-vectors satisfy the Kruskal-Katona inequalities: those with few vertices relative to their dimension. Our starting point is the following lemma, see \cite{Mesh} and \cite[Lemma 2.1.14]{Gal}. (Recall that the boundary of the $d$-dimensional \emph{cross-polytope} is the $d$-fold join of the zero-dimensional sphere, called also the \emph{octahedral sphere}.)

\begin{lemma}[Meshulam, Gal]\label{lem:Gal}
If $\Delta$ is a flag homology sphere then:
\begin{enumerate}
\item[(a)] $\gamma_1(\Delta)\geq 0$,
\item[(b)] if $\gamma_1(\Delta) = 0$, then $\Delta$ is an octahedral sphere.
\end{enumerate}
\end{lemma}

By definition, if $\Delta$ is a $(d-1)$-dimensional flag homology sphere, we have $f_1(\Delta) = 2d + \gamma_1(\Delta)$. For Theorem \ref{thm:smallGamma1} we will classify $\gamma$-vectors of those $\Delta$ for which $0 \leq \gamma_1(\Delta) \leq 3$, or equivalently $2d \leq f_1(\Delta) \leq 2d+3$. Notice that an octahedral sphere (of any dimension) has $\gamma = (1,0,0,\ldots)$.

If $\Delta$ is a flag homology $d$-sphere, $F\in\Delta$ and $|F|=k$, then $\lk(F)$ is a flag homology $(d-k)$-sphere (for flagness see Lemma \ref{lem:flag-facts}(b) below). The \emph{contraction} of the edge $\{u,v\}$ in $\Delta$ is the complex $\Delta' = \{ F \in \Delta : u \notin F \} \cup \{ (F \setminus \{u\})\cup \{v\} : F \in \Delta, u \in F\}$. By \cite[Theorem 1.4]{Nevo} $\Delta'$ is a sphere if $\Delta$ is a sphere, but it is not necessarily flag. The same holds for homology spheres \cite[Proposition 2.3]{Nevo-Novinsky}.  

We have the following relation of $\gamma$-polynomials:
\begin{equation}\label{eq:con}
 \gamma(\Delta; t) = \gamma(\Delta';t) + t\gamma(\lk(\{u,v\});t).
\end{equation}
%%%%%%%%%%%%%%%%%%%%%%%%%%%%%%%%%%%%%%%%%
Also, the \emph{suspension} $\susp(\Delta) = \Delta \cup \{ \{a\} \cup F, \{b\} \cup F: F \in \Delta\}$ (for vertices $a$ and $b$ not in the vertex set of $\Delta$), of a flag sphere $\Delta$ has the same $\gamma$-polynomial as $\Delta$:
$$\gamma(\susp(\Delta);t) = \gamma(\Delta;t).$$
Further, for $A\subseteq V$, define $\Delta[A]$ to be the \emph{induced subcomplex} of $\Delta$ on $A$, consisting of all faces $F$ of $\Delta$ such that $F\subseteq A$. The \emph{antistar} $\antist(v)$ of a vertex $v\in V$ is the induced subcomplex $\Delta[V-\{v\}]$.

In the following lemma we collect some known facts and some simple observations which will be used frequently in what follows.

\begin{lemma}\label{lem:flag-facts}
Let $\Delta$ be a flag complex on  vertex set $V$. Then the following holds:

(a) If $A\subseteq V$ then $\Delta[A]$ is flag.

(b) If $F\in \Delta$ then $\lk(F)$ is an induced subcomplex of $\Delta$, hence flag.

%(c) Let $\Delta'$ be obtained from $\Delta$ by contracting an edge $\{u,v\}$. If $\Delta'$ is not flag then $\{u,v\}$ is contained in an induced $4$-cycle in $\Delta$.

Let $K$ be a simplicial complex on vertex set $U$ and $\Gamma$ a subcomplex of $K$ on vertex set $A$. Then:

(c) If $\Gamma=K[A]$ then $K-\Gamma$ deformation retracts on $K[U-A]$.

(d) If $K$ and $\Gamma$ are homology spheres then $K-\Gamma$ has the same homology as a sphere of dimension $\dim K - \dim \Gamma - 1$. In particular, if $\dim K=\dim \Gamma$ then $K=\Gamma$.
\end{lemma}
\begin{proof}
Part (a) is obvious. For (b), Let $v$ be a vertex of $\Delta$. If all proper subsets of a face $T\in \Delta$ are in $\lk(v)$, then by flagness $T\cup\{v\}\in \Delta$, hence $T\in \lk(v)$ so $\lk(v)$ is an induced subcomplex. If $F=T\cup\{v\}$ in $\Delta$ where $v\notin T$, then $\lk_{\Delta}(F) = \lk_{\lk(v)}(T)$, and by induction on the number of vertices in $F$ we conclude that $\lk(F)$ is an induced subcomplex. By part (a) it is flag, concluding (b).

%To prove (c), suppose $\Delta'$ is not flag and let $F$ be a minimal set such that all proper subsets of $F$ are in $\Delta'$ but $F\notin \Delta'$. Then by flagness of $\Delta$ there must exist two vertices $v'\neq u'$ in $F$ such that $\{v,v'\}\notin \Delta$ and $\{u,u'\}\notin \Delta$. Then $(v,u,v',u')$ is an induced $4$-cycle in $\Delta$, proving (c).

Part (c) is easy and well known, and (d) is a consequence of Alexander duality.
\end{proof}

It is clear that the link of any vertex in an octahedral sphere is itself an octahedral sphere. The following lemma, suggested to us by one of the referees and used in the sequel, shows that the converse is true as well.

\begin{lemma}\label{lem:octahedron}
Let $\Delta$ be a $(d-1)$-dimensional flag homology sphere on vertex set $V$ such that for any $v\in V$ $\lk(v)$ is an octahedral sphere. Then $\Delta$ is an octahedral sphere. 
\end{lemma}
\begin{proof}
Fix $v\in V$ and let $I$ be the set of interior vertices in the homology ball $\antist(v)$, i.e., the set of vertices that do not share an edge with $v$.
%By Lemma \ref{lem:flag-facts}(b) with $F=\{v\}$, $|I|>0$.
If $I$ is empty, then $\Delta$ is a cone over $\lk(v)$, which contradicts the fact that $\Delta$ is a homology sphere.
If $|I|$=1, say $I=\{u\}$, then $\lk(u)\subseteq \lk(v)$ are homology spheres of the same dimension, hence by Lemma \ref{lem:flag-facts}(d) $\lk(u)=\lk(v)$. Thus $\Delta$ contains the suspension of $\lk(v)$ and again by Lemma \ref{lem:flag-facts}(d) $\Delta=\susp(\lk(v))$. Thus $\Delta$ is octahedral.

Now assume for a contradiction that $|I|>1$. Then there exists a vertex $w\in \lk(v)$ with at least two neighbors in $I$,  say $a$ and $b$. Then $\lk_{\Delta}(w)$ contains the vertices in $\lk_{\lk_{\Delta}(v)}(w)$ and $\{a,b,v\}$, hence more then $2(d-1)$ vertices. But $\lk(w)$ is a $(d-2)$-dimensional octahedral sphere, so it has precisely $2(d-1)$ vertices. This is a contradiction.
\end{proof}

\begin{proposition}\label{prop:1}
If $\Delta$ is a $(d-1)$-dimensional flag homology sphere with $\gamma_1(\Delta)=1$, then $\gamma(\Delta;t) = 1+t$ and $\Delta$ is a repeated suspension over the boundary of a pentagon.
\end{proposition}

\begin{proof}
We will proceed by induction on dimension. As a base case $d=2$, observe that for $\Delta$ the boundary of an $n$-gon one has $f(\Delta;t) = 1+nt+nt^2$, $h(\Delta;t) = 1+(n-2)t+t^2$, and hence $\gamma(\Delta;t) = 1+(n-4)t$ and $\gamma_1(\Delta)=1$ only for the pentagon.

Now suppose $\Delta$ is a $(d-1)$-dimensional flag homology sphere with $2d+1$ vertices. If $\Delta$ is a suspension, it is a suspension of a homology $(d-2)$-sphere with $2d-1 = 2(d-1)+1$ vertices and we are finished by induction. 

Otherwise, the link of any vertex $v$ is a $(d-2)$-dimensional homology sphere with precisely $2d-2$ vertices, i.e. an octahedral sphere. By Lemma \ref{lem:octahedron} $\Delta$ is an octahedral sphere, so this case is impossible.
\end{proof}

\begin{proposition}\label{prop:2}
If $\Delta$ is a $(d-1)$-dimensional flag homology sphere with $\gamma_1(\Delta)=2$, then $\gamma(\Delta;t) \in \{ 1+2t, 1+2t+t^2\}$.
\end{proposition}

\begin{proof}
Again we proceed by induction on dimension. For base case $d=2$, as we observed beforehand $\gamma_1(\Delta)=2$ only for the boundary of a hexagon, in which case
 $\gamma(\Delta;t) = 1+2t$. Assume $d>2$.

We now analyze the structure of $\Delta$ according to the number of vertices in the interior of the antistar of a vertex $v\in \Delta$, denoted by $i(v)$. We always have $i(v)>0$ as $\Delta$ is flag with nontrivial top homology (use Lemma \ref{lem:flag-facts}(b) with $F=\{v\}$).

If there is a vertex $v\in \Delta$ with $i(v)=1$, then $\Delta$ is the suspension over $\lk(v)$, and we are done by induction on dimension.

If there is a vertex $v\in \Delta$ with $i(v)=2$, let $b$ and $c$ denote the vertices in the interior of its antistar. First we show that $\{b,c\}\in \Delta$:
assume by contradiction that  $\{b,c\}$ is not an edge in $\Delta$. Then the homology $(d-2)$-sphere $\lk(b)$ must be contained in the induced subcomplex $\lk(v)$ and by Lemma \ref{lem:flag-facts}(d) we get $\lk(b) = \lk(v)$. Deleting $c$ gives a proper subcomplex of $\Delta$ that is itself a homology $(d-1)$-sphere (the suspension over $\lk(v)$), an impossibility. Thus $\{b,c\}$ must be an edge in $\Delta$.

Let $\Delta'$ be obtained from $\Delta$ by contracting the edge $\{b,c\}$. Then $\gamma_1(\Delta') = 1$. Since $\Delta'$ is also a flag homology sphere (it is the suspension over $\lk(v)$) we have $\gamma(\Delta')=1+t$ by Proposition \ref{prop:1}.
We now show that
$\gamma_1(\lk(\{b,c\}))\in\{0,1\}$. Let $m=\gamma_1(\lk(\{b,c\}))$ and 
assume by contradiction that $m\geq 2$.  
By Lemma \ref{lem:flag-facts}(b) $\lk(\{b,c\})$ is an induced subcomplex of codimension $1$ in $\lk(v)$, and by \ref{lem:flag-facts}(d)  $\lk(v)-\lk(\{b,c\})$ is homologous to the zero dimensional sphere. Thus $\lk(v)$ has at least $2$ vertices more than $\lk(\{b,c\})$, hence $\Delta$ has at least $5$ vertices more than $\lk(\{b,c\})$. This means $\gamma_1(\Delta)\geq m+1\geq 3$, a contradiction.
 
Thus $\gamma(\lk(\{b,c\})\in \{1,1+t\}$ and by \eqref{eq:con}, $\gamma(\Delta)\in\{1+2t,1+2t+t^2\}$ in this case.

The last case to consider is when $i(v)=3$ for every vertex $v\in \Delta$. In this case, a vertex count tells us any $\lk(v)$ is an octahedral sphere, hence by Lemma \ref{lem:octahedron} $\Delta$ is an octahedral sphere, so this case is impossible. 
\end{proof}

\begin{proposition}\label{prop:3}
If $\Delta$ is a $(d-1)$-dimensional flag homology sphere with $\gamma_1(\Delta)=3$, then $\gamma(\Delta;t) \in \{ 1+3t, 1+3t+t^2, 1+3t+2t^2, 1+3t+3t^2+t^3\}$.
\end{proposition}
\begin{proof}
Again we proceed by induction on dimension. For base cases $d=2,3$  
$\gamma(\Delta;t) = 1+3t$ and there is nothing to prove. Assume $d>3$.

As in the proof of Proposition \ref{prop:2}, we fix a vertex $v\in \Delta$ and analyze the structure of $\Delta$ according to the number $i(v)>0$.
If $i(v)=1$ then $\Delta$ is the suspension over $\lk(v)$, and we are done by induction on dimension.
  
If $i(v)=2$, then as before we conclude that $\{b,c\}\in \Delta$, where $b$ and $c$ are the vertices in the interior of the antistar of $v$ in $\Delta$, and for $\Delta'$ obtained from $\Delta$ by contracting the edge $\{b,c\}$ observe that $\Delta'$ is the flag homology sphere which is the suspension over $\lk(v)$. Thus $\gamma(\Delta')=\gamma(\lk(v))\in\{1+2t, 1+2t+t^2\}$ by Proposition \ref{prop:2}.  

As in the proof of Proposition \ref{prop:2}, $\gamma_1(\lk(\{b,c\}))\leq \gamma_1(\Delta)-1$, so in this case $\gamma_1(\lk(\{b,c\})) \leq 2$. Thus by Propositions \ref{prop:1} and \ref{prop:2}, $\gamma(\lk(\{b,c\}))\in \{1, 1+t, 1+2t, 1+2t+t^2\}$.  By (\ref{eq:con}), to conclude the assertion we need to show that the two cases where one of $\gamma(\Delta')$ and $\gamma(\lk(\{b,c\}))$ equals $1+2t$ and the other equals $1+2t+t^2$ are impossible. In these two cases $\gamma_1(\lk(\{b,c\}))=\gamma_1(\lk(v))=2$. As $\lk(\{b,c\})$ is an induced homology sphere of codimension 1
in $\lk(v)$, by Lemma \ref{lem:flag-facts}(d) 
it separates $\lk(v)$ and we conclude that $\lk(v)=\susp(\lk(\{b,c\}))$, hence $\gamma(\lk(\{b,c\}))=\gamma(\lk(v))=\gamma(\Delta')$, showing the above two cases are impossible.

If $i(u)=4$ for every vertex $u\in \Delta$, then all vertex links are octahedral spheres, an impossibility by Lemma \ref{lem:octahedron}. 

We are left to deal with the case where for every vertex $u\in \Delta$, $i(u)\geq 3$ and there exists a vertex $v\in \Delta$ with $i(v)=3$.
Let $I(v)=\{b,c,e\}$ be the set of interior vertices in $\antist(v)$.
By \ref{lem:flag-facts}(c) and (d)
the induced subcomplex $\Delta[v,b,c,e]$ is homotopic to $\Delta-\lk(v)$ and hence homologous to the zero dimensional sphere. Thus,  $\Delta[b,c,e]$ is either a triangle or a $3$-path, say $(b,c,e)$. 

If $\Delta[b,c,e]$ is a triangle, let $F$ be a facet in $\antist(v)$ containing it and $x$ a vertex in $F \cap \lk(v)$. %Then $\gamma_1(\lk(x))\geq 2$, i.e., $i(x)<3$, so this case is impossible.
We see $\lk_{\lk(v)}(x)$ is a $(d-3)$-flag homology sphere and so has at least $2d-4$ vertices by Lemma \ref{lem:Gal}. But then $\lk_{\Delta}(x)$ has at least $2d$ vertices (now counting $b, c, e$, and $v$). Thus since $\Delta$ itself has $2d+3$ vertices, $x$ can have at most $2$ vertices in its antistar. This contradicts the assumption that $i(x) \geq 3$.

Now suppose $\Delta[b,c,e]$ is the $3$-path $(b,c,e)$.
By Proposition \ref{prop:1}, $\lk(v)$ is a repeated suspension over a pentagon. Denote the pentagon by $C$.
The argument we used in the case of a triangle shows that we can assume that a vertex  $x\in \lk(v)$ is contained in $\lk(\{b,c\})$ only if $x\in C$ (as otherwise $i(x)<3$). Thus $\lk(\{b,c\})\subseteq C$, hence the dimension of $\Delta$ is at most $3$. Thus  
$\gamma(\Delta)$ satisfies $0=\gamma_3=\gamma_4=\ldots$ and $0\leq \gamma_2\leq \lfloor\frac{\gamma_1^2}{4}\rfloor=2$. The assertion follows.  
\end{proof}
%%%%% Constructions %%%%%
 
To complete the proof of Theorem \ref{thm:smallGamma1} we construct a flag sphere for each admissible $\gamma$-vector. Let $C_m$ denote the $m$-gon and $*$ the simplicial join operation. As mentioned before, a $\gamma$-vector of the form $(1,m,0,0,\ldots)$ is $\gamma(C_{m+4})$, $m\geq0$.
Recall that the $\gamma$-polynomial is multiplicative with respect to join.
Then $\gamma(C_5*C_5;t)=(1+t)^2=1+2t+t^2$, $\gamma(C_5*C_6;t)=(1+t)(1+2t)=1+3t+2t^2$ and $\gamma(C_5*C_5*C_5;t)=(1+t)^3=1+3t+3t^2+t^3$. 
Lastly, let $\Delta$ be obtained from $C_5*C_5$ by subdividing an edge whose vertices belong to different copies of $C_5$. By \eqref{eq:con} (see also \cite[Proposition 2.4.3]{Gal}) we get $\gamma(\Delta;t)=(1+t)^2 + t = 1+3t+t^2$.
$\square$

%%%%%%%%%%%
%%%%%%%%%%%
\section{Stronger inequalities}\label{sec:FFK}

A $(d-1)$-dimensional simplicial complex $\Delta$ on a vertex set $V$ is \emph{balanced} if there is a coloring of its vertices $c: V\to [d]$  such that for every face $F\in \Delta$ the restriction map $c: F\to [d]$ is injective. That is, every face has distinctly colored vertices.

Frohmader \cite{Frohmader} proved that the $f$-vectors of flag complexes form a (proper) subset of the $f$-vectors of balanced complexes. (This was conjectured earlier by Eckhoff and Kalai, independently.) Further, a characterization of the $f$-vectors of balanced complexes is known \cite{FFK}, yielding stronger upper bounds on $f_{i+1}$ in terms of $f_i$ than the Kruskal-Katona inequalities, namely the \emph{Frankl-F{\"u}redi-Kalai inequalities}. For example, a balanced $1$-dimensional complex is a bipartite graph, hence satisfies $f_2\leq f_1^2 /4$, while the complete graph has $f_2=\binom{f_1}{2}$. See \cite{FFK} for the general description of the Frankl-F{\"u}redi-Kalai inequalities.

Because the $\Gamma$-complexes of Section \ref{sec:Gam} are flag complexes, Frohmader's result shows that the $\gamma$-vectors of Theorem \ref{thm:KrKex} satisfy the Frankl-F{\"u}redi-Kalai inequalities. The same is easily verified for the $\gamma$-vectors given by Theorem \ref{thm:smallGamma1} and in Table \ref{table} for the exceptional Coxeter complexes. We obtain the following strengthening of Theorem \ref{thm:KrKex}.
%\kyle{I still need to check the exceptional cases!}

\begin{theorem}\label{thm:FFK}
The $\gamma$-vector of $\Delta$ satisfies the Frankl-F{\"u}redi-Kalai inequalities for each of the following classes of flag homology spheres:
\begin{itemize}
\item[(a)] $\Delta$ is a Coxeter complex.

\item[(b)] $\Delta$ is the simplicial complex dual to an associahedron.

\item[(c)] $\Delta$ is the simplicial complex dual to a cyclohedron.

\item[(d)] $\Delta$ has $\gamma_1(\Delta) \leq 3$.
\end{itemize}

\end{theorem}

%\eran{remark added:}
\begin{remark}
The complexes $\Gamma(S)$ where $S \in \{\Des_n, \Pk_n, \Des_n^D\}$ 
are balanced. The color of a vertex $v$ with a peak at position $i$ is $\lceil \frac{i}{2}\rceil$.
\end{remark}

Similarly this suggests the following strengthening of Conjecture \ref{conj:flagKrK}.

\begin{conjecture}\label{conj:flagFFK}
If $\Delta$ is a flag homology sphere then $\gamma(\Delta)$ satisfies the  Frankl-F{\"u}redi-Kalai inequalities.
\end{conjecture}

As mentioned in the Introduction, this conjecture is true for flag homology spheres of dimension at most $4$. 
%However, while closer to the truth than Conjecture \ref{conj:flagKrK}, it is still not sharp. In particular, the vector $(1,4,5,1)$ satisfies the Frankl-F\"uredi-Kalai inequalities, but it is possible using techniques as in Section \ref{sec:small} that there is no flag homology sphere with this $\gamma$-vector. It happens that $(1,4,5,1)$ falls in the gap between the set of $f$-vectors of flag complexes and the set of $f$-vectors of balanced complexes. That is, there is no flag complex whose $f$-vector is $(1,4,5,1)$. 
We do not have a counterexample to the following strengthening of this
conjecture.
%, which may still be sharp.

\begin{problem}\label{prob:flagFlag}
If $\Delta$ is a flag homology sphere, then $\gamma(\Delta)$ is the $f$-vector of a flag complex.
\end{problem}
Very recently Frohmader (personal communication) verified that the $\gamma$-vectors of the exceptional Coxeter complexes are the $f$-vector of flag complexes, by straightforward `greedy' constructions. 

%\begin{acknowledgements} 
{\it{Acknowledgements.}}
We thank the referees for helpful suggestions.
%\end{acknowledgements}

\end{document}